\documentclass[10pt, a4paper]{article}
\usepackage[top=3cm, bottom=3cm, left=3cm, right=3cm]{geometry}
\date{}
\usepackage[utf8]{inputenc}
\usepackage[english]{babel}
\usepackage{amsthm}
\usepackage{amsmath}
\usepackage{amsfonts}
\usepackage[onehalfspacing]{setspace}
\usepackage{amssymb}
\usepackage{bbold}
\newtheorem{theorem}{Theorem}[section]
\usepackage{amsthm}
\usepackage{mathtools}
\usepackage{bbm}
\usepackage{url}

\def\wh{\widehat}
\def\wt{\widetilde}

\newtheorem{lemma}{Lemma}[section]

\newtheorem{proposition}{Proposition}[section]

\newtheorem*{theorem*}{Theorem}
\theoremstyle{remark}
\newtheorem*{remark}{Remark}

\newtheorem*{lemma*}{Lemma}
\newtheorem*{proposition*}{Proposition}
\usepackage{tikz-cd}
\usepackage{mathrsfs}
\usepackage{hyperref}
\usepackage{fancyhdr}
\theoremstyle{remark}

\title{The algebraicity of generating functions using nested Artin approximation}

\begin{document}
\author{Gregor Böhm\footnote{Supported by the Austrian Science Fund, FWF, Project PAT4965924\\The author is grateful to Herwig Hauser for insightful suggestions and stimulating conversations. Furthermore we would like to thank Manuel Kauers for helpful discussions. During the preparation of this text, the author used Meta AI (Muse Spark 1.1, September 2026) to improve readability, refine formulations, and enhance the academic tone of the text.}}
\maketitle
\begin{abstract}
   We present an elementary proof of the bivariate nested Artin-Popescu approximation theorem: it ensures the existence of a nested algebraic power series solution of a given polynomial (functional) equation. Such equations often appear in enumerative combinatorics, e.g., when counting lattice walks in the first quadrant. The ad hoc proofs of the algebraicity of the generating functions, as e.g. proposed by Bousquet-Mélou and Jehanne in \cite{bousquetmélou2005polynomialequationscatalyticvariable}, can thus be replaced by applying the theorem without need to resort to the extremely difficult general case solved by Popescu.
\\Our method of proof follows closely the arguments of Denef-Lipshitz in \cite{Denef1980},
where a more general case is treated (namely, the Artin approximation for Weierstrass systems). This is combined with the techniques of Hauser-Woblistin developed in \cite{Hauser_Woblistin_2021} for the description of the overall geometry of the infinite dimensional variety formed by all power series solutions. Putting both approaches together now provides combinatorialists with an accessible proof for the bivariate nested Artin-Popescu approximation theorem.
\end{abstract}
\section{Introduction}
Let $\mathbb{K}$ be a field of characteristic zero. The theorem for which this paper gives a proof is the following:
\begin{theorem}\label{thm1}
    Let $P(t,x,y,z)=0$ be a polynomial equation in some single variables $t,x,z,y$ with a nested formal power series solution $(\wh{y}(t),\wh{z}(t,x)).$ Assume that the partial derivative of $P$ with respect to $z$ satisfies $$\frac{\partial P}{\partial z}(t,x,\wh{y}(t),\wh{z}(t,x)) \not\equiv 0 \bmod (t),$$ i.e., it has a monomial not depending on the variable $t$. Then there exists an algebraic power series solution $(y(t),z(t,x)),$ $$P(t,x,y(t),z(t,x))=0,$$ with the same nestedness property.
    \\In particular, if $P=0$ has a unique solution this solution is already algebraic.
\end{theorem}
       Various generalizations and extensions are possible: 
        \begin{itemize}
        \item First, $y$ and $z$ may be multivariables. If $x$ is a single variable and one has the condition $$\frac{\partial P}{\partial z}(t,x,\wh{y}(t),\wh{z}(t,x)) \not\equiv 0 \bmod (t)$$ $t$ may be a multivariable as well. Conversely, if $x$ is a multivariable, the proof technique requires $t$ to be a single variable with the extra ingrediant of the classical N\'{e}ron desingularization (see e.g. \cite{strongartin1}, \cite{Denef1980}). This machinery is also required if one drops the assumption $\frac{\partial P}{\partial z}(t,x,\wh{y}(t),\wh{z}(t,x)) \not\equiv 0 \bmod (t)$.
        \item One may require that the algebraic solution approximates the formal one to an arbitrary prescribed degree, yielding the so-called approximation property.\\
        \item The single polynomial $P$ can be replaced by a vector of polynomials (or even algebraic power series), with essentially the same proof, but more technicalities. We refer to section \ref{sec gen} for more details on all this. \\
         \item If one wants to generalize the problem to an arbitrary number of nests with arbitrary size this culminates in the theorem of Popescu on nested Artin Approximation using his deep structure theorem about regular homomorphisms between noetherian rings \cite{popescu_1986}, \cite{Spivakovsky}, \cite{swan}, \cite{OGOMA199457}, \cite{stacks-projectpop}, \cite{SB_1993-1994__36__259_0}.
        \item In many applications (see e.g. \cite{bousquetmélou2005polynomialequationscatalyticvariable}), the components $y_i(t)$ of the formal and algebraic solutions $y(t)=(y_0(t),\ldots,y_m(t))$ are the first $m$ coefficients of the Taylor expansion at zero of $z(t,x)$ as a series of $x$. This is an extra condition which is covered by the theorem by adding additional equations.\\
        \end{itemize}
  We emphasize that the proof given in this article is mostly extracted from Denef-Lipshitz' paper \cite{Denef1980}, using also techniques from \cite{Hauser_Woblistin_2021}, while omitting all complications and restricting always to the simplest but still relevant case. In their paper Denef and Lipshitz actually treat the case where $P$ is convergent and the solutions are asked to be convergent as well. However, the proof of the algebraic case is essentially the same using the algebraic Weierstrass division instead of the convergent Weierstrass division.\\ 
  \\The strategy of the proof we present is tricky. Given the formal solution $\wh{y}(t)$ and $\wh{z}(t,x)$, one first concentrates on the bivariate second component $\wh{z}(t,x)$, while treating the univariate $\wh{y}(t)$ as a stowaway. By an argument proposed by Denef and Lipshitz one is then able to eliminate completely the $x$-variable from the problem, at the cost of introducing auxiliary formal series $\wh{u}(t)=(\wh{u}_1(t),...,\wh{u}_s(t))$ and a whole new system of algebraic equations for the series $\wh{y}(t)$ and $\wh{u}(t)$. But, in this case, the classical univariate Artin approximation theorem, say, Greenberg's theorem \cite{Greenberg1966}, applies and ensures the existence of algebraic solutions. Resubstitution of the algebraic solutions $y(t)$ and $u(t)$ into the original problem then allows one to reconstruct the second component $z(t,x)$ of the solution as an algebraic series, and the pair $(y(t), z(t,x))$ gives the desired total algebraic nested solution of $P(t,x,y,z) = 0$.\\

\section{An introductory example}
In this section, we give, in line with \cite{bousquetmélou2005polynomialequationscatalyticvariable}, an easy example of a classical generating functon that arises in combinatorics for which Theorem\ref{thm1} gives algebraicity: Generalized Dyck paths. We want to count the number of walks in $\mathbb{N}^2$ with $n$ steps of the form $(1,1)$ and $(1,-1)$. Let $d_{n,k}$  be the number of such walks with $n$ steps ending in $(n,k)$ and let $z(t,x)\coloneqq\sum d_{n,k} t^n x^k$ be its generating function. One easily sees that the generating function satisfies the equation
\begin{equation}\label{dyck}   
z(t,x)=1+txz(t,x)+\frac{t}{x}\left(z(t,x)-z(t,0)\right),\end{equation}
One arrives at this formula by considering the last step of a walk. The empty path contributes the $1$. If the last step taken was $(1,1)$ one gets the second term. The third term arises from the case where the last step is $(1,-1)$ taking into account that at this point $k$ cannot be zero. Furthermore, the case $k=0$ is well known as the Dyck paths with endpoint on the $x$-axis and thus $d_{2n,0}$ is the $n$-th Catalan number. Once again considering the cases of the last step and solving the arising functional equation yields $$z(t,0)=\frac{1-\sqrt{1-4t^2}}{2t^2}.$$ This is clearly an algebraic function as it satisfies the equation
$$4t^4z(t,0)^2-4t^2z(t,0)+4t^2=0.$$ However, one is interested whether $z(t,x)$ is algebraic as well. Theorem \ref{thm1} shows that this is the case: Rewriting equation (\ref{dyck}) yields that $z(t,x)$ and $y(t)=z(t,0)$ is a nested solution of $P=0$ for the equation $$P(t,x,y,z)=x+tx^2z+t(z-y)-xz=0.$$ Now notice that for any other nested solution $\widetilde{z}(t,x),\widetilde{y}(t)$ setting $x=0$ yields $\widetilde{y}(t)=\widetilde{z}(t,0)$. Then $z(t,x),z(t,0)$ is the unique nested solution by construction. Furthermore $\frac{\partial P}{\partial z}=tx^2+tz-x$ and then egardless of the form of $z(t,x)$ the term $tx^2+tz(t,x)$ has at least order $1$ in $t$ and cannot cancel the $x$ term. Thus, the assumption of Theorem \ref{thm1} is satisfied and $z(t,x)$ is algebraic. 
\\This is a special instance of a so-called polynomial equation in one catalytic variable $x$ (in the notation of \cite{ZEILBERGER2000451}). Theorem \ref{thm1} yields algebraicity of the solution of all such polynomial equations in one catalytic variable. For more concrete examples see e.g. \cite{bousquetmélou2005polynomialequationscatalyticvariable}.
\section{The proof of Theorem \ref{thm1}}\label{sec2}
Denote by $\mathbb{K}\langle x\rangle$ the ring of algebraic power series over $\mathbb{K}$ (i.e., the power series $f\in \mathbb{K}[[x]]$ that satisfy a nonzero polynomial equation $P(f(x))=0$, for $P\in\mathbb{K}[x,y]$). Furthermore, we will introduce another notation: For an algebraic equation $P(t,x,y,z)=0$ we will write $\wh{z}(t,x)$ and $\wh{y}_0(t),\ldots,\wh{y}_m(t)$ for a given formal nested solution to distinguish it from later found algebraic solutions $z(t,x)$ and $y_0(t),\ldots,y_m(t).$ We will furthermore repeatedly use the Weierstrass Division, the Implicit Function Theorem and the notion of $x$-regularity in the rest of this note. For the exact statements, see the \nameref{apendix}. 
\\Before we give the proof of Theorem \ref{thm1} we will state and prove a lemma that shows a basic fact which we will use twice in the rest of the paper:
\begin{lemma}[Hauser-Woblistin \cite{Hauser_Woblistin_2021}]\label{ideale}
    Let $g(x,y)\in\mathbb{K}[[x,y]]$ be a formal power series in two sets of variables $x=(x_1,\ldots,x_n)$ and $y=(y_1,\ldots,y_m)$. Let $y(x)\in\mathbb{K}[[x]]^m$ be a formal power series such that $$g(x,y(x))\neq0$$ and assume it is $x_n$-regular of order $d$.
    Divide formally $$y(x)=v(x)\cdot g(x,y(x))+r(x),$$ for some divisor $v\in\mathbb{K}[[x]]^m$ and remainder $r\in\mathbb{K}[[x_1,\ldots,x_{n-1}]][x_n]_{\leq d}^m.$ Then the two ideals $I,J$ generated by $g(x,y(x))$ and $g(x,r(x))$ respectively are the same. In particular, $y(x)$ can be written as $$y(x)=w(x)g(x,r(x))+r(x)$$ for some $w\in\mathbb{K}[[x]]^m.$ 
\end{lemma}
\begin{remark}
    Note that given $g(x,y)$ this gives a bijection between the $y(x)$ and the $r(x).$ Given $y(x)$ one gets $r(x)$ as a remainder of dividing $y(x)$ by $g(x,y(x))$, and given $r(x)$ there exists a unique $w(x)$ such that $y(x)=w(x)g(x,r(x))+r(x)$.
\end{remark}
\begin{proof}
Let $v=(v_1,\ldots,v_m)$ and consider the Taylor expansion of $g(x,r(x))$ up to degree two:
\begin{equation}
g(x,r(x))=g(x,y(x))-g(x,y(x))\cdot v(x)\cdot\partial_{\tt{y}}g(x,y(x))+g(x,y(x))^2\cdot h(x),
\end{equation}
where $h(x)\in x\mathbb{K}[[x]]$. Now factor out $g(x,y(x))$ to get 
\begin{equation}    
g(x,r(x))=g(x,y(x))\left[1-v(x)\cdot\partial_{\tt{y}}g(x,y(x))+g(x,y(x))\cdot h(x)\right].
\end{equation}
Then $1+v(x)\cdot\partial_{\tt{y}}g(x,y(x))+g(x,y(x))\cdot h(x)$ is invertible and hence the two ideals $I,J$ are the same. 
\end{proof}
\noindent We will not directly prove Theorem \ref{thm1}, but rather give a proof of the following result. Theorem \ref{thm1} is an easy consequence.
\begin{theorem}[\cite{Denef1980}]\label{easiest pop}
Let $P(t,x,y,z)\in \mathbb{C}\langle t,x,y,z\rangle$ be an
algebraic power series in three single variables $t,x,z$ and some multivariable $y=(y_0,\ldots,y_m)$. Assume given formal power series vectors $\widehat{y}(t),\wh{z}(t,x)$ vanishing at $0$ (i.e., $\wh{y}\in (t)\mathbb{C}[[t]]^{m+1},\wh{z}\in (t,x)\mathbb{C}[[t,x]]$), such that
$$P(t,x,\wh{y}(t),\wh{z}(t,x))= 0.$$
Assume that $$\frac{\partial P}{\partial z}(t,x,\widehat{y}(t),\widehat{z}(t,x))\neq0\bmod (t).$$
Then there exists, for any $c \in \mathbb{N}$, algebraic
power series vectors $y(t)\in \mathbb{C}\langle t\rangle^{m+1},z(t,x)\in\mathbb{C}\langle t,x\rangle,$ with $$P(t,x, y(t)
,z(t,x)) = 0,$$
which coincide with $\wh{y}(t),\wh{z}(t,x)$ up to degree $c$,
$$y(t)\equiv \wh{y}(t) \: \bmod \:(t)^{c+1},$$
$$z(t,x)\equiv \wh{z}(t,x) \: \bmod \:(t,x)^{c+1}.$$
\end{theorem}
\noindent As mentioned above this is a special case of Popescu's nested approximation theorem but allows for a simpler proof. The theorem can easily be generalized to multiple $z$ variables, which is done in section \ref{sec gen}. 
\\ \textit{Outline of the proof of Theorem \ref{easiest pop}:} The proof of Theorem \ref{easiest pop} has two steps. 
\\For the first step, one introduces a ring that emphasizes the $x$-variable. This ring is used to express the series $\wh{z}(t,x)$ as a power series that is algebraic in $x$ and in certain formal power series $\wh{u}(t)=(\wh{u}_1(t),\ldots,\wh{u}_s(t))$ in $t$. The idea is to then approximate both $\wh{y}(t)$ and $\wh{u}(t)$ by algebraic power series using the Artin Approximation Theorem in the second step. To do so, we consider the following ring introduced by \cite{Denef1980}:
$$\mathbb{K}[[t\vert x\rangle\coloneqq\{q(\wh{u}_1(t),\ldots,\wh{u}_s(t),x)\in\mathbb{K}[[t,x]],\;q(u_1,\ldots,u_s,x)\in \mathbb{K}\langle u_1,\ldots, u_s,x \rangle,$$$$  \wh{u}_1(t),\ldots, \wh{u}_s(t)\in (t)\mathbb{K}[[t]]\quad \mathrm{for}\;\mathrm{some}\;s\in \mathbb{N}\}.$$
Its elements are obtained from algebraic power series in which the variables are replaced by formal power series in $t$. Note that $s$ is not fixed. We note that we will also use the same notation for $x$ being a multivariable later on. One easily verifies that this is a ring and inherits the following properties of the rings $\mathbb{K}\langle u_1,\ldots,u_s,x\rangle$:
\begin{lemma}
Let $R=\mathbb{K}[[t\vert x\rangle$. Then
    \begin{itemize}
        \item $R$ is closed under partial differentiation with respect to the $x_i$,
        \item the Weierstrass Division theorem holds in $R$,
        \item the Inverse and Implicit Function Theorem hold for functions in $R$.
    \end{itemize}
\end{lemma}
 \noindent The next step is to prove the following generalization of the Implicit Function Theorem:
\begin{theorem}[Generalized Implicit Function Theorem]\label{easier ''}
    Let $x$ be a single variable and $Q(t,x,z)\in\mathbb{K}[[t\vert x,z\rangle$ with a formal power series solution $\wh{z}(t,x)\in (t,x)\mathbb{K}[[t,x]]$ of $Q=0$ such that $$\frac{\partial Q}{\partial z}(t,x,\wh{z}(t,x))\neq0\bmod (t).$$ Then $\wh{z}(t,x)\in \mathbb{K}[[t\vert x\rangle$.   
\end{theorem}
 \noindent Note that this is indeed a generalization of the classical Implicit Function Theorem, as it does not consider the derivative at zero but at the formal solution $\wh{z}(t,x)$. To illustrate this generalization consider once again the equation for generalized Dyck paths: $$P(t,x,y,z)=x+tx^2z+t(z-y)-xz$$ and its derivative with respect to $z:$
$$\frac{\partial P}{\partial z}=tx^2+tz-x.$$
Here the conventional Implicit Function Theorem does not apply as $\frac{\partial P}{\partial z}(0)=0 $. However the generalized version holds as $tx^2+tz(t)-x\neq0\bmod (t)$ for all formal power series $z(t)$.
\\The statement only holds for $z$ a single variable, for generalizations see section \ref{sec gen}. The following proof was suggested to us by Herwig Hauser.
\begin{proof}[Proof of Theorem \ref{easier ''}]
 The basic idea of the proof is to reduce to the case where the usual Implicit Function Theorem applies. 
 \\Set $M(t,x,\wh{z}(t,x))\coloneqq\frac{\partial Q}{\partial z}(t,x,\wh{z}(t,x))\neq0\bmod (t)$ by assumption. The power series $M(t,x,\wh{z}(t,x))$ is therefore $x$-regular of order $d$ for some $d\in\mathbb{N}$. We now use formal Weierstrass Division to get $$\wh{z}(t,x)=\wh{v}(t,x)M(t,x,\wh{z}(t,x))+\widehat{r}(t,x),$$
 where $\wh{v}(t,x)\in\mathbb{K}[[t,x]]$ and $\widehat{r}(t,x)=\sum_{i=0}^{d-1}\widehat{r}_i(t)x^i\in\mathbb{K}[[t]][x]_{<d}$. From now on we will often leave out the $x$ and $t$ in $\widehat{z}(t,x),\wh{v}(t,x)$ and $\widehat{r}(t,x)$ for readability reasons.
 \\The remainder $\wh{r}$ is polynomial in $x$ and hence has finitely many formal power series in $t$ as coefficients.  Hence, it already is an element of $\mathbb{K}[[t\vert x\rangle$:
 Rewrite, using Lemma \ref{ideale},   $$\wh{z}(t,x)=\wh{w}(t,x)M(t,x,\wh{r}(t,x))+\widehat{r}(t,x),$$
for some $\wh{w}(t,x)\in\mathbb{K}[[t,x]]$.
Write $M(t,x,\wh{r})=x^{d}\wt{M}(t,x,\wh{r})$, with $\wt{M}(t,x,\wh{r})=\lambda+h(t,x,\wh{r})$ for some $\lambda\in\mathbb{K}^*$ and some $h(t,x,r)\in(t,x,\kappa)\mathbb{K}[[t\vert x,\kappa\rangle$. Taylor expansion yields
\begin{align}\label{simple root 1}
\begin{split}
    0=Q(t,x,\wh{z})=Q(t,x,\wh{r}+x^d\cdot \wh{w}\cdot \widetilde{M}(t,x,\wh{r}))\\
    =Q(t,x,\wh{r})+x^d\cdot \wh{w}\cdot \widetilde{M}(t,x,\wh{r})\cdot\partial_zQ(t,x,\wh{r})+S(t,x,x^d\cdot \wh{w}),
\end{split}
\end{align}
where $S(t,x,w)\in\mathbb{K}[[t\vert x,w\rangle$ with $\mathrm{ord}_w S\geq2$. From \begin{align*}
    \mathrm{ord}_x(x^d\cdot\wh{w}\cdot \widetilde{M}(t,x,\wh{r})\cdot\partial_zQ(t,x,\wh{r}))&=\mathrm{ord}_x(x^{2d}\cdot \wh{w}\cdot \widetilde{M}(t,x,\wh{r})^2)=2d+\mathrm{ord}_x\wh{w},\\
    \mathrm{ord}_x(S(t,x,x^d\cdot \wh{w}))&\geq2d,
\end{align*}
it follows that
$$\mathrm{ord}_xQ(t,x,\wh{r})\geq 2d.$$
Now divide equation (\ref{simple root 1}) by $x^{2d}$ to get
$$0=x^{-2d}Q(t,x,\wh{r})+x^{-2d}S(t,x,x^d\cdot \wh{w})+\wh{w}\cdot(\lambda+h(t,x,\wh{r}))^2.$$
This is an implicit equation for $w$ with the assumption of the classical Implicit Function Theorem. Therefore, this equation has a unique solution in $\mathbb{K}[[t\vert x\rangle$. Then $\wh{w}$ is this unique solution and belongs to $\mathbb{K}[[t\vert x\rangle$. From this it follows that $\wh{z}\in\mathbb{K}[[t\vert x\rangle$ as $\mathbb{K}[[t\vert x\rangle$ is a ring.
\end{proof}
\noindent This Generalized Implicit Function Theorem will be used in combination with:
\begin{theorem}[Univariate Artin Approximation \cite{Greenberg1966},\cite{ARTIN1968}]\label{ArtinAprox}
Let $t$ denote a single variable and $y=(y_1,\ldots,y_m)$ a multivariable. Further, let $f(t, y)\in \mathbb{K}\langle t,y\rangle^r$ be a vector of algebraic power series in $t$ and $y$. Assume given a formal power series solution $\widehat{y}(t)\in\mathbb{K}[[t]]^m$ of $f(t,y)=0$ vanishing at $0$,$$f(t,\widehat{y}(t))=0.$$ Then there exists, for any $c \in \mathbb{N}$, an algebraic power series solution $y(t)\in \mathbb{K}\langle t\rangle^m$,$$f(t, y(t)) = 0,$$ which coincides with $\widehat{y}(t)$ up to degree $c$, $$y(t)\equiv \widehat{y}(t) \: \bmod \:(t)^{c+1}.$$
\end{theorem}
\noindent The proof of this theorem is not very involved and only uses some commutative algebra, Taylor expansion and the Weierstrass Division theorem. It will be given in section \ref{artin section}. We note that in our setting one actually has to consider the case $r\geq 1$ as we will want to approximate both $\wh{y}(t)$ and $\wh{u}(t)$ at the same time.
\\With it we finally have all the ingredients to  prove Theorem \ref{easiest pop}:
\begin{proof}[Proof of Theorem \ref{easiest pop}]
Let $P(t,x,y,z)$ be an algebraic power series with a nested formal solution $\wh{y}(t)\in(t)\mathbb{K}[[t]]^{m+1},\;\wh{z}(t,x)\in(t,x)\mathbb{K}[[t,x]]$ of $P=0$. Consider $Q(t,x,z)\coloneqq P(t,x,\wh{y}(t),z)$. This series is an element of $\mathbb{K}[[t\vert x,z\rangle$. Furthermore, $Q=0$ is an equation with formal solution $\wh{z}(t,x)$ satisfying $\frac{\partial Q}{\partial z}(t,x,\wh{z}(t,x))\neq0\bmod (t).$ Hence, $\widehat{z}(t,x)\in\mathbb{K}[[t\vert x\rangle$ by the generalized Implicit Function Theorem.
\\By construction $\wh{z}(t,x)$ can be written as $\sum_{i\geq 0}h_{i}(\wh{u}(t))x^i,$ where $\sum_{i\geq 0}h_{i}(u)x^i \in\mathbb{K}\langle u,x\rangle$ and $\wh{u}(t)=(\wh{u}_1(t),\ldots,\wh{u}_s(t))\in (t)\mathbb{K}[[t]]$. Substituting this $\wh{z}(t,x)$ into $P(t,x,y,z)$ yields $$0=P(t,x,\wh{y}(t),\sum_{i\geq 0}h_i(\wh{u}(t))x^i).$$ Expansion as a power series in $x$ gives $$0=\sum_{i\geq0}H_i(t,\wh{y}(t),\wh{u}(t))x^i,$$ where the $H_i$ are elements of $\mathbb{K}\langle t,y,u\rangle$ since $P$ is algebraic. 
 As the ring $\mathbb{K}\langle t,y,u\rangle$ is noetherian, the infinite ideal $( H_i \mid i\in\mathbb{N})$ can be reduced to an equivalent finite system $\{H_j\mid j\in J\}$ for some finite set $J$ such that all $H_j$ vanish at the formal solution $\wh{y}(t),\wh{u}(t)$. Now apply the classical Artin Approximation Theorem for some $c\in\mathbb{N}$ to get algebraic $(y(t),u(t))\in\mathbb{K}\langle t\rangle^{m+s}$ such that $H_j(t,y(t),u(t))=0$ for all $j\in J$. Therefore, $H_i(t,y(t),u(t))=0$ for all $i\in\mathbb{N}$. 
 Furthermore, 
 $$y_i(t)\equiv\wh{y}_i(t) \: \bmod \:(t)^{c+1}\: \mathrm{for} \: i=1,\ldots,m,$$
     $$u_j(t)\equiv\wh{u}_j(t) \: \bmod \:(t)^{c+1}\: \mathrm{for} \: j=1,\ldots,s.$$
     Set $z_i(t,x)=\sum_{i\geq 0}h_{i}(u(t))x^i$. Then, by construction, $z(t,x)$ is an algebraic power series with $P(t,x,y(t),z(t,x))=0$ and $$z_i(t,x)\equiv \wh{z}_i(t,x) \: \bmod \:(t,x)^{c+1}\: \mathrm{for} \: i=1,\ldots,k.$$
     \\Hence, $z(t,x)$ and $y(t)$ are as required.
\end{proof}
\begin{remark}
    Hauser and Woblistin \cite{Hauser_Woblistin_2021} prove that the univariate solution space of solutions $y(t),u(t)$ of the system of equations $H_j(t,y,u)$ for $j\in J$ admits a natural stratification such that every stratum decomposes into a cartesian product of a finite dimensional variety (that may be non-smooth) and some infinite dimensional affine space. Therefore, after substitution into $\sum_{i\geq 0}h_{i}(u)x^i$, the space of nested solutions obtained by this proof technique is the image of a Cartesian product. More specifically, it is the image of a Cartesian product where one factor is finite dimensional and may be singular, whereas the second factor is an infinite dimensional affine space.
\end{remark}
\section{The proof of the univariate Artin Approximation Theorem}\label{artin section}
For the convenience of the reader this section offers a proof of the  univariate Artin Approximation Theorem. Just like in the proof of Theorem \ref{easier ''} above, we try to reduce to the case where the Implicit Function Theorem applies. So we start with an algebraic power series vector $f(t,y)\in\mathbb{K}\langle t,y\rangle^r$ and a formal solution vector $\wh{y}(t)\in(t)\mathbb{K}[[t]]^m$ of $f(t,y(t))=0$. We want to divide by a $t$- regular minor of $\left(\frac{\partial f_i}{\partial y_j}\right)_{i,j}$ evaluated in $\wh{y}(t)$. So we will first have to show that such a minor always exists: 
\begin{lemma}\label{Reduction}
    Let $f(t, y)\in \mathbb{K}\langle t,y\rangle ^r$ be a vector of algebraic power series in $t$ and $y=(y_0,\ldots,y_m)$. Assume given a formal power series solution $\widehat{y}(t)$ vanishing at $0$,$$f(t,\widehat{y}(t))=0.$$
    Then there exists a system of algebraic equations $F=(F_1(t,y),\ldots,F_s(t,y))\in\mathbb{K}\langle t,y\rangle^s$ with an $s\times s$-minor $M$ of $\partial_{\tt{y}} F$ that does not vanish on $\wh{y}(t)$ (i.e., $M(t,\wh{y}(t))\neq 0)$ and such that if $y(t)$ is an algebraic solution of $F=0$ approximating $\wh{y}(t)$ up to an arbitrarily high degree $c$, then $y(t)$ is already a solution of $f=0$.
\end{lemma}
\begin{proof}
    Let $I$ be the ideal generated by $(f_1,\ldots,f_r)$ in $\mathbb{K}\langle t,y\rangle$. We may replace $I$ by the kernel of the substitution homomorphism $\mathbb{K}\langle t,y \rangle \rightarrow \mathbb{K}[[t]],$ sending $y$ to $\wh{y}(t)$. Therefore, we may assume that $I$ is prime. Let $s$ be the height of $I$. In other words, if $Y$ denotes the analytic germ $(\mathbb{K}^{n+m},0)$ and $X=V(I)$, then $s=\mathrm{codim}_Y(X)$, as $I$ being prime implies that all the components of $V(I)$ have the same dimension.
    We now use the following proposition:
    \begin{proposition}[\cite{Mumford} p.44 Cor.4]
        Let $U$ be an affine variety and $Z$ a closed irreducible subset and let $s=$codim$_U(Z)$. Then there exists a regular sequence $F_1,\ldots,F_s$ such that $Z$ is a component, of the complete intersection $V(F_1,\ldots,F_s)$.
    \end{proposition}
    \noindent We can therefore assume that $X$ is an irreducible component of a complete intersection variety $X^*$ and that there exists an ideal $I^*\subseteq I$ generated by a regular sequence $F=(F_1,\ldots,F_s)$ of length $s$ of elements of $I$ such that the variety defined by $I$ is an irreducible prime component of the variety defined by $I^*$. Let $J$ be the intersection of the other prime components of $I^*$. By the Jacobian Criterion (\cite{Mumford} p.168 Cor.1 to Prop. 2) there exists an $(s\times s)$-minor $M$ of the Jacobian $DF$ such that $M$ is not in $I^*$ meaning $M(t,\widehat{y}(t))\neq 0$. The partial derivative of $0=F(t,\widehat{y}(t))$ with respect to $t$ is $$0=\partial_{t}F(t,\widehat{y}(t))+\partial_{t}\widehat{y}\,\cdot\,\partial_{\tt{y}}F(t,\widehat{y}(t)),$$
equivalently 
$$\partial_{t}F(t,\widehat{y}(t))=-\partial_{t}\widehat{y}\,\cdot\,\partial_{\tt{y}}F(t,\widehat{y}(t)).$$
We may therefore assume that $M$ can be chosen within the relative Jacobian $\partial_{\tt{y}}F$ with respect to the $y$-variables. If $X=X^*$, we have already found the minor $M$ we were looking for.
\\Now assume that $X\neq X^*$ and $\widetilde{y}(t)$ is an algebraic solution of $F$ that approximates $\widehat{y}(t)$ up to a chosen degree. Our goal is to show that $\widetilde{y}(t)$ is also a solution of $f$ for all $f\in I$. As $X\neq X^*$, there is $h(t,y)\in K\langle t,y\rangle$ that does not vanish on $V(I)$ but on $V(J)$. Therefore, $h(t,\widehat{y}(t))\neq 0$ and as $\widetilde{y}(t)$ approximates $\widehat{y}(t)$ to a high degree, also $h(t,\widetilde{y}(t))\neq 0$.
\\Now consider the map $\Pi:\mathbb{K}\langle t,y \rangle \rightarrow \mathbb{K}\langle t\rangle,$sending $y$ to $\wh{y}(t)$ and let $L=ker(\Pi)$. Then $L$ is a prime ideal, as $K\langle t \rangle$ is an integral domain. However, as $\widetilde{y}(t)$ is a solution of $F_1,\ldots,F_s$ we get 
$$I*J\subseteq I\cap J =I^*\subseteq L .$$
As $h(t,\widetilde{y}(t))\neq 0$ we have $J\nsubseteq L$ and hence, as $L$ is Prime, $I\subseteq L$. Therefore, $f(t,\widetilde{y}(t))=0$ for all $f\in I$.
\end{proof}
\noindent Using this, we can now prove the univariate Artin Approximation Theorem:
\begin{proof}[Proof of Theorem \ref{ArtinAprox}]
    The idea is once again to reduce to a case where the Implicit Function Theorem can be applied. We may assume, using Lemma \ref{Reduction}, that there exists an $r\times r$-Minor $M$ of $\left(\frac{\partial f_i}{\partial y_j}\right)_{i,j}$ such that $M(t,\wh{y}(t))\neq0$. We denote by $y^\prime$ the set $\{y_j\mid M\; \mathrm{involves}\;\frac{\partial f}{\partial y_j},\;1\leq j\leq m\}$ and let $y^{\prime\prime}$ denote the set $\{y_j\mid y_j\not\in y^\prime,\;1\leq j\leq m\}.$ As $M(t,\wh{y}(t))$ is nonzero it is $t$-regular of order $d$ for some $d\in\mathbb{N}$. Hence, we can apply the formal Weierstrass division:
   \begin{equation}\begin{gathered}
\wh{y}_i(t) = v_i(t)\cdot M(t,\wh{y}(t)) + r_i(t) \quad \mathrm{for} \;\mathtt{y}_i\in \mathtt{y}^\prime,\\
\wh{y}_i(t) = v_i(t)\cdot M(t,\wh{y}(t))^2 + r_i(t) \quad \mathrm{for} \;\mathtt{y}_i\in \mathtt{y}^{\prime\prime},
\end{gathered}\end{equation}
for some formal power series $v_i\in\mathbb{K}[[t]]$ and $r_i(t)\in\mathbb{K}[t]_{\leq2d}$ for $1\leq i\leq m$. We will once again leave out the $t$ arguments of $r(t),\wh{y}(t)$ and $v(t)$ for the rest of the proof. Using Lemma \ref{ideale} the division can be rewritten as 
\begin{equation}
\begin{gathered}
    \wh{y}_i = \wh{a}_i(t)\cdot M(t,r) + r_i \quad \mathrm{for} \;\mathtt{y}_i\in \mathtt{y}^\prime,\\
\wh{y}_i = \wh{a}_i(t)\cdot M(t,r)^2 + r_i \quad \mathrm{for} \;\mathtt{y}_i\in \mathtt{y}^{\prime\prime},
\end{gathered}
\end{equation}
for some $\wh{a}(t)\in\mathbb{K}[[t]]^m$ for which we will leave out the $t$ in the future as well.
\\Let $\wh{a}^\prime$ be the vector of all $\wh{a}_i$ for which the corresponding $\mathtt{y}_i$ belongs to $ \mathtt{y}^\prime$ and let $\wh{a}^{\prime\prime}$ denote the vector of those $\wh{a}_i$ for which the corresponding $\mathtt{y}_i\in \mathtt{y}^{\prime\prime}$ for $1\leq i\leq m$. Furthermore, let $\wh{a}=(\wh{a}^\prime,M(t,r)\wh{a}^{\prime\prime})$. Now consider the Taylor expansion of $f(t,\wh{y})=f(t,\wh{a}\cdot M(t,r)+r)$:
\begin{equation}\label{2.1}
    f(t,r)+\wh{a}\cdot M(t,r)\cdot\partial_{\tt{y}}f(t,r)+M(t,r)^2\,\cdot\, q(t,\wh{a},M(t,r))=0,
\end{equation}
where $q(t,\omega,\kappa)$ is an algebraic power series in $t$, $\omega=(\omega^\prime,\omega^{\prime \prime})$ and $\kappa$ and $q$ is at least quadratic in $\omega$. As $M=\partial_{\tt{y}^\prime}f$ is a maximal minor of $\partial_{\tt{y}} f$, we can write the diagonal matrix $M\,\cdot\, \mathbb{1}_r$ as the product of $\partial_{\tt{y}^\prime}f$ and its adjoint matrix $\partial_{\tt{y}^\prime}^*f$. We can use this fact to get
\begin{equation}\label{2.2}
\begin{gathered}
M(t,r)^2\,\cdot\, q(t,\wh{a},M(t,r))=\\
M(t,r)\,\cdot\,\partial_{\tt{y}^\prime}f(t,r)\,\cdot\,\partial_{\tt{y}_\prime}^*f(t,r)\,\cdot\, q(t,\wh{a},M(t,r))
\end{gathered}
\end{equation}
and
\begin{equation}\label{2.3}
\begin{gathered}
\wh{a}\cdot M(t,r)\cdot\partial_{\tt{y}}f(t,r)=\\
M(t,r)\,\cdot\,[\sum_{\mathtt{y}_i\in\mathtt{y}^\prime}\partial_{\tt{y}_{i}}f(t,r)\,\cdot\, \wh{a}_{i}+\sum_{\mathtt{y}_i\in\mathtt{y}^{\prime\prime}}\partial_{\tt{y}_{i}}f(t,r)\,\cdot\,\wh{a}_{i}\,\cdot M(t,r)]=\\
M(t,r)\,\cdot\,\partial_{\tt{y}^\prime}f(t,r)\cdot[\wh{a}^\prime+\partial_{\tt{y}{\prime \prime}}f(t,r)\,\cdot\,\wh{a}^{\prime \prime}\,\cdot\,\partial_{\tt{y}_\prime}^*f(t,r)].
\end{gathered}
\end{equation}

\noindent Combine this with (\ref{2.1}) to get
\begin{equation}
f(t,r)+M(t,r)\,\cdot\,\partial_{\tt{y}^\prime}f(t,r)\,\cdot\,\Phi(t,\wh{a}^\prime,\wh{a}^{\prime\prime},r)=0,
\end{equation}
where \begin{equation}
\Phi(t,\omega^\prime,\omega^{\prime\prime},\kappa)=(\omega^\prime+\partial_{\tt{y}^\prime}^*f(t,\kappa)\,\cdot\, q(t,\omega^\prime,M(t,\kappa)\cdot \omega^{\prime\prime},M(t,\kappa))+\partial_{\tt{y}^\prime}^*f(t,\kappa)\,\cdot\,\partial_{\tt{y}^{\prime \prime}}f(t,\kappa)\,\cdot\,\omega^{\prime\prime}.
\end{equation}
\noindent Then $\Phi$ is algebraic, vanishing in zero, and $\partial_{\omega^\prime}\Phi(0,0,0,0)$ does not vanish and is therefore invertible.
Multitplication with $\partial_{\tt{y}^\prime}^*f$ yields that $\partial_{\tt{y}^\prime}^*f(t,r)\cdot f(t,r)$ is in the ideal generated by $M(t,r)^2.$ As both $f(t,y)$ and $M(t,y)$ are algebraic power series and $r$ is polynomial $\Phi(t,\wh{a}^\prime,\wh{a}^{\prime\prime},r)$ has to be an algebraic power series.
\\Now let $c\in\mathbb{N}$ be a natural number and let $a^{\prime\prime}(t)\in\mathbb{K}[t]_{\leq c}$ be the $c$-jet of $\wh{a}^{\prime\prime}$. We then consider the equation $$\Psi(t,\omega^\prime)=\Phi(t,\wh{a}^\prime,\wh{a}^{\prime\prime},r)-\Phi(t,\omega^\prime,a^{\prime\prime},r).$$

\noindent Then $\Psi(\omega,a^\prime)$ is an algebraic equation vanishing in zero satisfying $$\partial_{\omega^\prime}\Psi(0,0)=\partial_{\omega^\prime}\Phi(0,0,0,0)\neq0.$$
Thus, by the Implicit Function Theorem there exist $a^\prime(t)\in\mathbb{K}\langle t\rangle$ with $\Phi(t,\wh{a}^\prime,\wh{a}^{\prime\prime},r)=\Phi(t,a^\prime,a^{\prime\prime},r)$ and therefore $$f(t,r)+M(t,r)\,\cdot\,\partial_{\tt{y}^\prime}f(t,r)\,\cdot\,\Phi(t,a^\prime,a^{\prime\prime},r)=0.$$
Setting $a(t)=\left(a^\prime(t),M(t,r)a^{\prime\prime}(t)\right)$ and repeating (\ref{2.1}) and (\ref{2.2}) in reverse order yields  $$f(t,r)+a\cdot M(t,r)\cdot\partial_{\tt{y}}f(t,r)+M(t,r)^2\,\cdot\, q(t,a,M(t,r))=0.$$
This is the Taylor expansion of $f(t,a(t)\cdot M(t,r)+r)$. Thus, $y(t)\coloneqq a\cdot M(t,r)+r$ is a solution of $f$ that is algebraic as it is a sum and product of algebraic power series. Also by construction $y(t)\equiv\wh{y}(t)\bmod (t)^c$.
\end{proof}
\section{On generalizations of Theorem \ref{thm1}}\label{sec gen}
In this chapter, we will discuss the generalizations mentioned in the introduction. 
The assumption of a finite amount of solutions instead of a unique solution can be handled rather easily: As there are only finitely many nested solutions, there exists a $c\in\mathbb{N}$ such that $\wh{z}(t,x)\bmod(t,x)^{c+1}$ already determines $\wh{z}(t,x)$ as a unique solution.
\\The theorem of section \ref{sec2} can be generalized to systems of equations and vectors of $z$ variables:
\begin{theorem}[\cite{Denef1980}]
Let $P(t,x,y,z)\in \mathbb{C}\langle t,x,y,z\rangle^r$ be a system of
algebraic power series in two single variables $t,x$ and some multivariables $y=(y_0,\ldots,y_m),z=(z_1,\ldots,z_k)$. Assume given a formal power series solution $\widehat{y}(t),\wh{z}(t,x)$ vanishing at $0$,
$$P(t,x,\wh{y}(t),\wh{z}(t,x))= 0,$$
such that there exists a maximal minor $M(t,x,y,z)$ of $\frac{\partial P_i}{\partial z_j}$ with $$M(t,x,\widehat{y}(t),\widehat{z}(t,x)) \neq0\bmod (t).$$
Then there exists, for any $c \in \mathbb{N}$, an algebraic
power series solution $y(t)\in \mathbb{C}\langle t\rangle^{m+1},z(t,x)\in\mathbb{C}\langle t,x\rangle^k$,$$P(t,x, y(t)
,z(t,x)) = 0,$$
which coincides with $\wh{y}(t),\wh{z}(t,x)$ up to degree $c$,
$$y(t)\equiv \wh{y}(t) \: \bmod \:(t)^{c+1}$$
$$z(t,x)\equiv \wh{z}(t,x) \: \bmod \:(t,x)^{c+1}.$$
\end{theorem}
\noindent The proof is analogous to the proof in chapter \ref{sec2}. However, in this case the Generalized Implicit Function Theorem  is not enough. We replace it with an approximation theorem:
\begin{proposition}[\cite{Denef1980}]
  Let $Q(t,x,z)\in\mathbb{K}[[t\vert x,z\rangle^r$ with a formal power series  solution $\wh{z}(x)\in (t,x)\mathbb{K}[[t]][[x]]^k$ of $Q(t,x,\wh{z}(x))=0$ such that there exists a maximal minor $\widetilde{M}$ of $\frac{\partial Q_i}{\partial z_j}$ such that $\widetilde{M}(t,x,\wh{z}(t,x))\neq0\bmod (t).$ Then there exists for every $c\in \mathbb{N}$ a solution $z(x)\in \mathbb{K}[[t\vert x\rangle^k$ agreeing with $\wh{z}$ up to degree $c$.      
\end{proposition}
\noindent The proof of this proposition is essentially the same as the proof of the Artin Approximation Theorem given in chapter 4. 
\\Finally, we show that the assumption of $\wh{y}_i(t)$ being the coefficient of $x^i$ in the series expansion of $\wh{z}(t,x)$ for $0\leq i\leq m$ can be expressed in terms of $(m+1)$ new polynomial equations in $t,x,y,z$ and $w=(w_0,\ldots,w_m)$ where the Jacobi matrix with respect to the $\tt z$ and the $\tt w$ variable has as determinant exactly $\frac{\partial P}{\partial z}$.
\\We consider:
\begin{align}
\begin{split}\label{system}
    P(t,x,y,z)&=0,\\
    z-y_0+w_0&=0,\\
    z-y_0-x\cdot y_1+w_1&=0,\\
    &\;\;\vdots\\
    z-\sum_{i=0}^mx^i\cdot y_i+w_m&=0.         
\end{split}
\end{align}
We set $\wh{w}_j(t,x) \coloneqq -\wh{z}(t,x)+\sum_{i=0}^jx^i\cdot \wh{y}_i(t)$. By assumption $\wh{w}_j(t,x)\in x^{j+1}\mathbb{K}[[t,x]]$ for $j=0,\ldots,m$.
Now one gets approximating algebraic solutions $y(t),z(t,x)$ and $w(t,x)$
with $w_j(t,x)=-z(t,x)+\sum_{i=0}^jx^i\cdot y_i(t)$ for $0\leq j\leq m$. For $c\in\mathbb{N}\geq m$ we have $\wh{w}_j(t,x)\equiv w_j(t,x)\bmod (x)$ and thus $w_j(t,x)\in x^{j+1}\mathbb{K}[[t,x]]$. Thus, $y_j(t)$ is the coefficient of $x^j$ in the series expansion of $z(t,x)$ for $0\leq j\leq m$.
\section{Appendix}\label{apendix}
In this appendix, consider the ring of algebraic power series $\mathbb{K}\langle x \rangle$ in $x=(x_1,\ldots,x_n)$.
\\An element $h(x)$ of $\mathbb{K}\langle x \rangle$ is called $x_n$-regular of order $d$ if $h(0,\ldots,0,x_n)=x_n^d u(x_n)$, where $u(x)\in \mathbb{K}\langle x_n\rangle ^*$ is a unit. In other words, the monomial $x_n^d$ appears in the series expansion of $h(x)$ and $d$ is the smallest such degree. Then the Weierstrass Division Theorem can be stated as the following:
\begin{theorem}[Weierstrass Division Theorem]\label{WDT}
Let $f\in \mathbb{K}\langle x \rangle$ be $x_n$-regular of order $d$. Then for any $g\in \mathbb{K}\langle x \rangle$ there exist unique power series $h\in \mathbb{K}\langle x \rangle$ and $r_0,\ldots,r_{d-1} \in \mathbb{K}\langle x^\prime \rangle$ where $x^\prime=(x_1,\ldots,x_{n-1})$ such that 
$$g(x)=f(x)\cdot h(x)+\sum_{j=0}^{d-1} r_j(x^\prime)\cdot x_n^j.$$
\end{theorem}
\noindent The proof goes back to \cite{Laf}.
\\The Weierstrass Division Theorem implies the Implicit Function Theorem and Inverse Function Theorem which are equivalent. We know of no proof of the Implicit function theorem implying the Weierstrass Division Theorem.
\begin{theorem}[Inverse Function Theorem]
    Let $f=(f_1,\ldots,f_n) \in \mathbb{K}\langle x \rangle^n$ such that $f(0)=0$ and the Jacobian matrix $\partial_xf$ of $f$ is invertible in $0$. Then there exists a unique $h\in \mathbb{K}\langle x \rangle^n$ vanishing at $0$ such that $h \circ f=f \circ h=\mathrm{Id}_{\mathbb{K}\langle x \rangle}$.
\end{theorem}
\begin{theorem}[Implicit Function Theorem]
  Let $g=(g_1,\ldots,g_m)\in \mathbb{K}\langle x,y \rangle^m$ be a vector of power series in $x=(x_1,\ldots,x_n)$ and $y=(y_1,\ldots,y_m)$ with $g(0,0)=0$ such that $\partial_yg$ is invertible at $0$. Then there exists a unique $h\in \mathbb{K}\langle x \rangle^m$ vanishing at $0$, such that $g(x,h(x))=0$.
\end{theorem}
\noindent The proof of their equivalence is classical, while the ideas of Weierstrass Division implying the Inverse Function Theorem goes back to \cite{doi:10.1080/00927879008824067}.
\fontsize{8}{9}\selectfont
\bibliographystyle{alpha}
\bibliography{references}

@book{Mumford,
  title     = "The Red Book of Varieties and Schemes",
  author    = "D. Mumford",
 series = "Lecture
Notes in Mathematics",
Volume = "1358",
  year      = "1999",
  publisher = "Springer"
}

@article{ARTIN1968,
author = {M. Artin},
journal = {Inventiones mathematicae},
pages = {277-291},
title = {On the Solutions of Analytic Equations.},
url = {http://eudml.org/doc/141922},
volume = {5},
year = {1968},
}

@article{popescu_1986, title={General {N}éron desingularization and approximation}, volume={104}, DOI={10.1017/S0027763000022698}, journal={Nagoya Mathematical Journal}, publisher={Cambridge University Press}, author={Popescu, D.}, year={1986}, pages={85–115}}

@article{Denef1980,
author = {Denef, J. and Lipshitz, L.},
journal = {Mathematische Annalen},
pages = {1-28},
title = {Ultraproducts and {A}pproximation in {L}ocal {R}ings. {II}},
url = {\url{http://eudml.org/doc/163462}},
volume = {253},
year = {1980},
}

@article{strongartin1,
author={Artin, M.},
title= {Algebraic approximation of structures over complete local rings.},
journal= {Publications Mathématiques de l'IHÉS},
volume ={36},
year={1969},
pages= {23-58},
url={\url{http://www.numdam.org/item/PMIHES_1969__36__23_0/}}
}

@article{Greenberg1966,
author = {Greenberg, M.},
journal = {Publications Mathématiques de l'IHÉS},
language = {eng},
pages = {59-64},
publisher = {Institut des Hautes Études Scientifiques},
title = {Rational points in henselian discrete valuation rings},
url = {http://eudml.org/doc/103871},
volume = {31},
year = {1966},
}

@article{Spivakovsky,
   title={A new proof of {D}. {P}opescu’s theorem
on smoothing of ring homomorphisms},
   volume={12},
   number={2},
   journal={{J}ournal of the American {M}athematical {S}ociety},
   author={M. Spivakovsky},
   year={1999},
   pages={381-444} }

@article {Laf,
  title={S\`{e}ries formelles alg\`{e}briques},
  author={Lafon J.-P.},
  journal={Comptes rendus de l’Académie des sciences},
  year={1965},
  volume={260},
  pages={3238-3241}
}

@article{doi:10.1080/00927879008824067,
author = {A. v. d. Essen},
title = {A criterion to decide if a polynomial map is invertible and to
compute the inverse},
journal = {Communications in Algebra},
volume = {18},
number = {10},
pages = {3183-3186},
year = {1990},
publisher = {Taylor & Francis},
doi = {10.1080/00927879008824067}
}

@article{bousquetmélou2005polynomialequationscatalyticvariable,
title = {Polynomial equations with one catalytic variable, algebraic series and map enumeration},
journal = {Journal of Combinatorial Theory, Series B},
volume = {96},
number = {5},
pages = {623-672},
year = {2006},
issn = {0095-8956},
doi = {https://doi.org/10.1016/j.jctb.2005.12.003},
url = {https://www.sciencedirect.com/science/article/pii/S0095895605001681},
author = {M. Bousquet-Mélou and A. Jehanne}
}

@article{swan,
    author = {R. Swan},
    title = {Néron-{P}opescu desingularization},
    journal ={Algebra and Geometry},
    year ={1995},
   adresse= {Taipei},
  note={Lect. Algebra Geom., vol. 2, Int. Press, Cambridge, 1998, pp. 135–192},
}

@article{OGOMA199457,
title = {General {N}éron {D}esingularization based on the idea of {P}opescu},
journal = {Journal of Algebra},
volume = {167},
number = {1},
pages = {57-84},
year = {1994},
issn = {0021-8693},
doi = {https://doi.org/10.1006/jabr.1994.1175},
url = {https://www.sciencedirect.com/science/article/pii/S0021869384711756},
author = {T. Ogoma}
}

@misc{stacks-projectpop,
  author       = {The {Stacks project authors}},
  title        = {The Stacks project},
 note = {tag 07BW},
  howpublished = {\url{https://stacks.math.columbia.edu/tag/07BW}},
  year         = {2024},
}

@incollection{SB_1993-1994__36__259_0,
     author = {Teissier, B.},
     title = {R\'esultats r\'ecents sur l'approximation des morphismes en alg\`ebre commutative},
     booktitle = {S\'eminaire Bourbaki : volume 1993/94, expos\'es 775-789},
     series = {Ast\'erisque},
     note = {talk:784},
     pages = {259--282},
     publisher = {Soci\'et\'e {M}ath\'ematique de France},
     number = {227,},
     year = {1995},
     mrnumber = {1321650},
     zbl = {0828.13003},
     language = {fr},
     url = {http://www.numdam.org/item/SB_1993-1994__36__259_0/}
}

@article{ZEILBERGER2000451,
title = {The {U}mbral {T}ransfer-{M}atrix {M}ethod. {I}. {F}oundations},
journal = {Journal of Combinatorial Theory, Series A},
volume = {91},
number = {1},
pages = {451-463},
year = {2000},
issn = {0097-3165},
doi = {https://doi.org/10.1006/jcta.2000.3110},
url = {https://www.sciencedirect.com/science/article/pii/S0097316500931108},
author = {D. Zeilberger}
}

@article{Hauser_Woblistin_2021, title={Arquile Varieties – Varieties Consisting of Power Series in a Single Variable}, volume={9}, DOI={10.1017/fms.2021.73}, journal={Forum of Mathematics, Sigma}, author={Hauser, H. and Woblistin, S.}, year={2021}, pages={e78}}

\quad {\scshape University of Vienna, Faculty of Mathematics, Oskar-Morgenstern-Platz 1, 1090, Vienna, Austria}
\\\textit{Email:} \href{gregor.pascal.boehm@univie.ac.at}{gregor.pascal.boehm@univie.ac.at} 
\end{document}